\documentclass[12pt]{amsart}
\usepackage{amsaddr}
\usepackage{latexsym}
\usepackage{a4}
\usepackage{amssymb}
\usepackage{setspace}
\usepackage{float}
\usepackage[english]{babel}
\usepackage{url}
\usepackage{tikz}
\usepackage{subcaption}
\usepackage{caption}
\allowdisplaybreaks
\newtheorem{thm}{Theorem}[section] 
\newtheorem{pro}[thm]{Proposition} 
\newtheorem{lem}[thm]{Lemma} 
 
\newtheorem{cor}[thm]{Corollary} 
\theoremstyle{definition}

\newtheorem{exa}[thm]{Example}
 
\newtheorem{de}[thm]{Definition}
\numberwithin{equation}{section} 

   \title{Closed function sets on groups of prime order}

\author{Sebastian Kreinecker}
\address{Sebastian Kreinecker,
Institute for Algebra,
Johannes Kepler University Linz,\\
Altenbergerstra{\ss}e 69,
4040 Linz,
Austria}
\email{\tt kreinecker@algebra.uni-linz.ac.at}

\thanks{Supported by the Austrian Science Fund (FWF):P29931}
\date{\today}

\newcommand{\N}{\mathbb{N}} 
\newcommand{\Z}{\mathbb{Z}} 
\newcommand{\ar}[1]{[#1]}
\newcommand{\al}[1]{\mathbf{#1}}
\newcommand{\size}[1]{\lvert #1 \rvert}
\newcommand{\Clo}[1]{\operatorname{Clo}(#1)}
\newcommand{\genLC}[1]{\langle #1 \rangle_{\operatorname{C}} }
\newcommand{\constFun}[1]{\mathcal{K}_{#1}}
\newcommand{\linFun}[1]{\mathcal{L}_{#1}}
\newcommand{\affFun}[1]{\mathcal{A}_{#1}}
\newcommand{\Pol}[1]{#1[X]}
\newcommand{\redPol}[1]{#1^*[X]}
\newcommand{\setDeg}[1]{\operatorname{DEGS}(#1)}
\newcommand{\setTotDeg}[1]{\operatorname{TD}(#1)}
\newcommand{\occVar}[1]{\operatorname{Var}(#1)} 
\newcommand{\maxInd}[1]{\operatorname{mI}(#1)}
\newcommand{\totD}[1]{{\operatorname{tD}}(#1)}
\newcommand{\pRep}[1]{\operatorname{rem}(#1)}
\newcommand{\genPLC}[1]{\langle #1 \rangle_{\operatorname{P}}}
\newcommand{\linPol}{\operatorname{L}}
\newcommand{\genLocEqCl}[1]{\operatorname{Loc}(#1)}
\newcommand{\Con}[1]{\operatorname{Con}(#1)}
\newcommand{\Polymorphisms}[1]{\operatorname{Pol}(#1)}
\newcommand{\aut}[1]{\operatorname{Aut}(#1)}
\newcommand{\pre}{\vartriangleright}
\newcommand{\ord}[1]{\operatorname{ord}(#1)}

\begin{document}
\bibliographystyle{amsplain}

\begin{abstract}
We give a full description of all sets of functions on the group $(\Z_p, +)$
of prime order which
are closed under the composition with the clone generated by $+$ from both sides.
Thereby, we also get a description of all iterative algebras on $\Z_p$ which 
are closed under the composition with the clone generated by $+$ from both sides.
As another consequence, there are infinitely many non finitely generated clones above $\Clo{\Z_p \times \Z_p, +}$ for $p > 2$.
\end{abstract} 

\maketitle

\section{Motivation and Introduction}

Starting from Emil Post's characterization of all closed sets of functions on a two-element set  \cite{P:PostLat},
there are many results describing sets of functions that are closed under certain properties.
This led to the study of clones \cite{PK:FUR}, iterative algebras \cite{H:VRL}, clonoids \cite{AM:FG}, 
minor closed sets \cite{P:GTM}
and other closed sets \cite{L:CCOF}.
Results for clones containing only affine mappings are given in 
\cite{S:COLO}, a full characterization for polynomial clones (clones containing the constants) on $\Z_p \times \Z_p$ and on $\Z_{p^2}$ for any prime $p$ can be found in \cite{B:PCCM}, 
a description for polynomial clones on $\Z_{pq}$ for different primes $p$, $q$ is given in \cite{AM:PCOG}
and polynomial clones on $\Z_n$ for $n$ squarefree are described in \cite{M:PC}.  
Let $\N := \{1,2,3,\ldots\}$ and $\N_0 := \N \cup \{0\}$.
Let $S$ be a nonempty set.
For $n\in\N$, we then denote by $S^{[n]}$ the set of $n$-ary operations on $S$.
We will express closure properties by means of the product defined in \cite{CF:FC} which we describe
for nonempty sets $A, B, C$, for $E \subseteq \bigcup \{ A^{B^n} \mid n\in \N \}$,
and for $F \subseteq \bigcup \{ B^{C^n} \mid n\in \N \}$ by 
 $$E F := \{f(g_1, \ldots, g_n) \mid n,m \in \N, f  \in E^{[n]}, g_1, \ldots, g_n \in F^{[m]}\}.$$
 For $n \in \N$ and for all $j \in \N$ with $j\leq n$ we define the 
 \emph{$j$-th projection} of $A^n$ by
 $\pi_j^n(x_1, \ldots, x_n):= x_j$ for $x_1, \ldots, x_n \in A$.
 Then $J_A := \{ \pi_j^n \mid n, j \in \N\;  \textmd{ with } j \leq n \}$
 is called the \emph{clone of projections}.
 An \emph{iterative algebra} on $A$ is a set $I \subseteq \bigcup \{ A^{A^n} \mid n\in \N \}$
 such that $I(I \cup J_A) \subseteq I$ (called \emph{semiclone} in \cite{B:GTFS}).
 If also $J_A \subseteq I$, we call $I$ a \emph{clone} on $A$ (cf.~\cite{S:CIUA}).

The goal of this paper is a characterization for sets of finitary functions of the following concept. 
Let $\al{A}$ be an algebra with universe $A$.
We call $C$ a \emph{clonoid with source set $B$ and target algebra $\al{A}$} 
 if $\Clo{\al{A}}(CJ_B) \subseteq C$ (cf. \cite{AM:FG}).
 Now let $p$ be a prime.
 We denote the \emph{set of linear functions on $\Z_p$}, which we define as the clone generated by the operation $+$, by $\linFun{\Z_p}$.
 We call a nonempty set $C$ of finitary operations on $\Z_p$ a \emph{linearly closed clonoid on $\Z_p$}  if
  $\linFun{\Z_p} C \subseteq C$ and $C \linFun{\Z_p} \subseteq C$. 
  The smallest linearly closed clonoid which contains two linearly closed clonoids $D_1$ and $D_2$ is given
  by $D_1 + D_2 := \{f+g \mid n\in \mathbb{N}, f\in D_1^{[n]}, g\in D_2^{[n]}\}$.
  The set of all linearly closed clonoids on $\mathbb{Z}_p$ together with the intersection
  and the sum forms a lattice.
 Now we introduce an important definition which will lead to our desired characterization.
 
   \begin{de}
  Let $M \subseteq \N_0$ and let $p$ be a prime.
  We call $M$ a \emph{$p$-minor subset} if $n-(p-1) \in M$ for all $n \in M$ with $n > p-1$. 
 \end{de}
 
 In Section \ref{sec:pminSub} we will prove the following main result of this paper which gives a full characterization of linearly closed clonoids on $\Z_p$:
 
   \begin{thm}\label{thm:ordIso}
 The lattice of $p$-minor subsets (under intersection and union) 
 is isomorphic to the lattice of linearly closed clonoids on $\Z_p$ (under intersection and sum).
 \end{thm}
 
 We call a linearly closed clonoid $C$ with $C(C \cup J_A )\subseteq C$
 a \emph{linearly closed iterative algebra}.
 In Section \ref{sec:resAndCons} we then will get a full description for linearly closed iterative algebras
 and for clones on $\Z_p$ which contain $+$.
 In Section \ref{sec:HagHer} we will show a different approach to get the same result for clones using a result of Hagemann and Herrmann \cite{HH:ALEC}.
 
  Hence we will find for a prime $p$ how many different linearly closed iterative algebras on $\Z_p$ there are 
  and also how many
  clones on $\Z_p$ contain $+$:
 
   \begin{cor}\label{cor:resultNoClones}
  Let $p$ be a prime number.
  Then there exist exactly $\lvert \{t \in \N \mid t \textmd{ divides } p-1 \} \rvert + 3$ 
  different clones on $\Z_p$ which contain $+$.
 \end{cor}

 We mention that Corollary \ref{cor:resultNoClones} is known from \cite{R:CPC}
 and it also follows from \cite[Theorem 22]{McK:PPV}.
By an embedding of linearly closed clonoids on $\Z_p$ into clones on $\Z_p \times \Z_p$,
we will show the following result in Section \ref{sec:ZpZp}:

  \begin{thm} \label{thm:infNonFinGenCl}
     Let $p > 2$ be a prime number.
     Then there are infinitely many non finitely generated clones on $\Z_p \times \Z_p$ which contain $+$. 
  \end{thm}

\section{Notations and Preliminaries}

In this section we fix some notation.

 \subsection{Important sets of functions}

 Let $p$ be a prime.
 We have already introduced the set of linear functions on $\Z_p$.
 Now we define the \emph{set of constant functions on $\Z_p$} by 
 $$\constFun{\Z_p} := \{\Z_p^n \to \Z_p, (x_1, \ldots, x_n) \mapsto a \mid n \in \N, a \in \Z_p\}.$$
 Furthermore, let $\affFun{\Z_p}$ be the \emph{set of affine functions on $\Z_p$}, that is,
 the clone generated by $\constFun{\Z_p}$ and the operation $+$.

 \subsection{Linearly closed clonoids}
 
 Let $p$ be a prime.
   Let $F \subseteq \bigcup \{ \Z_p^{\Z_p^n} \mid n\in \N \}$.
  We define $\genLC{F}$, the \emph{linearly closed clonoid on $\Z_p$ which is generated by $F$}, by 
  $\genLC{F} := \bigcap \{G \mid F \subseteq G, G \textmd{ is a linearly closed clonoid on $\Z_p$}\}$.

  Since on a field every finitary function is induced by a polynomial,
  it is sufficient to investigate polynomials and their properties.
 
 \subsection{Polynomials}
 
 Let $p$ be a prime.
  We fix an alphabet $X:= \{x_i \mid i \in \N\}$ and
 denote the set of polynomials on $\Z_p$ which are written in $X$ by $\Pol{\Z_p}$.
 Let $g \in \Pol{\Z_p}$.
 We define $\occVar{g}$ as the set of all variables which occur in $g$
 and by $\maxInd{g}$ we denote $\max(\{0\} \cup \{i \in \N \mid x_i \in \occVar{g} \})$.
  Let $0 \neq g\in \Pol{\Z_p}$.
  Then $g$
  can be written as
  $$g = \sum_{i=1}^n a_i \prod_{j = 1}^m x_j^{\alpha(i,j)}$$
  with the properties that for all $i \leq n$,
  $a_i \neq 0$, 
   and for all $i, i' \leq n$ with $i \neq i'$, there exists a $j\leq m$ such that $\alpha(i,j) \neq \alpha(i', j)$.
  For the monomial $h := \prod_{j=1}^m x_j^{\beta(j)}  \in \Pol{\Z_p}$
  we write $\mathbf{x}^{\boldsymbol{\beta}}$ where $\boldsymbol{\beta} = (\beta(1), \beta(2), \ldots, \beta(m))$ and
  for $g$ we then write
  $\sum_{i=1}^n a_i \mathbf{x}^{\boldsymbol{\alpha}(i)}$.
  We denote the \emph{total degree of a monomial $h$}, which is defined by the sum of the exponents, by $\totD{h}$.
   We denote the \emph{set of all degrees of $g$} by
  $$\setDeg{g} := \{\boldsymbol{\alpha}(i) \mid i \leq n\}$$
  and $\setDeg{0} := \emptyset$.
  We denote the \emph{set of
  all total degrees of $g$} by 
  $$\setTotDeg{g} := \{\sum_{j=1}^m \alpha(i,j) \mid i \leq n\}$$
  and $\setTotDeg{0} := \emptyset$.

  Now we introduce a notation for the composition of multivariate polynomials:
   Let $m \in \N$, $f_1, \ldots, f_m \in \Pol{\Z_p}$, and let $b = (b_1, \ldots, b_m ) \in \N^{m}$
   with $b_1 < b_2 < \ldots < b_m  $. Then we define
   \begin{multline*}
   g\circ_b (f_1, \ldots, f_m) :=\\ g(x_1, \ldots, x_{b_1-1}, f_1, x_{b_1 +1}, \ldots, x_{b_2-1}, f_2, x_{b_2+1}, \ldots,  x_{b_m-1}, f_m, x_{b_m+1}, \ldots).
   \end{multline*}
   Furthermore, we define
   $$g \circ_b f_1 := g\circ_b (\underbrace{f_1, f_1, \ldots, f_1}_{m} ).$$

   \emph{Examples}: $ (x_1 +x_2^2) \circ_{(1,2)} (x_1^3, x_1 +x_2) = x_1^3 + x_1^2 +2x_1x_2 +x_3^2$, and
   $(x_1) \circ_{(2)} (x_3) = x_1$.
  
  Let $d\in\N$.
  We write $\mathbf{1}_d$ for $(\underbrace{1, \ldots, 1}_{d})$.  
  
  Adapting our notations for the the product of two sets of 
  finitary functions
  we define now the product of $A, B \subseteq \Pol{\Z_p}$ (cf.~\cite{A:BFS}) by 
  $$AB := \{f (g_1, \ldots, g_n) \mid  n, m\in \N_0, \; f \in A^{\ar{n}},\;  g_1, \ldots, g_n \in B^{\ar{m}} \}.$$
  
  \emph{Remark:}
  A polynomial can be seen as a polynomial in any arity greater equal than its greatest variable.
  This means if $f\in \Pol{\Z_p}$, then for all $m \geq \maxInd{f}$, $f\in \Pol{\Z_p}^{\ar{m}}$.
  
  Now we define
  $$\linPol := \{ \sum_{i=1}^n a_i x_i \mid n\in \N, \forall i \leq n\colon a_i \in \Z_p \}. $$
  We call  $C \subseteq \Pol{\Z_p}$, nonempty, 
  a \emph{polynomial linearly closed clonoid} if 
  $\linPol C \subseteq C$ and $C\linPol \subseteq C$.
  Let $F \subseteq \Pol{\Z_p}$.
  Then we denote 
  by $\genPLC{F}$ the \emph{polynomial linearly closed clonoid
  that is generated by $F$}.  
  
    Now we show a fact about polynomial linearly closed clonoids and the composition of polynomials.
  
  \begin{lem} \label{lem:compPLCC}
  Let $C$ be a polynomial linearly closed clonoid and $g \in \Pol{\Z_p}$ 
  then the following hold:
  \begin{itemize}
   \item  For all $m \in \N$, for all $(b_1, \ldots, b_m) \in \N^{m}$, with $b_1 < \ldots < b_m$, for all $f_1 \ldots, f_m \in \linPol$, we have  
   $$g \circ_b (f_1, \ldots, f_m) \in C.$$
   As a special case we get:
   \item For all $m \in \N$, for all $b \in \N^{m} $, we have $g \circ_b 0 \in C$.
  \end{itemize}
  \end{lem}
  
  \begin{proof}  
  Let $\bar{m} = \maxInd{g}$.
  Let $n:= \max(\{\maxInd{f_i} \mid i\leq m  \})$. 
  Now we set $$(g_1, g_2, g_3, \ldots ) := (x_1, \ldots, x_{b_1-1 },f_1,x_{b_1+1}, \ldots, x_{b_m-1 },f_m,x_{b_m+1}, \ldots ),$$
  where $g_1, g_2, \ldots$ are polynomials of arity $n$.
  Then $$g(g_1, g_2, \ldots, g_{\bar{m}}) = g \circ_{(1, \ldots, \bar{m})} (g_1, g_2, \ldots, g_{\bar{m}}) = g \circ_b (f_1, \ldots, f_m)\in C,$$  
  since the left hand side lies in $C$ by the definition of a polynomial linearly closed clonoid.
  \end{proof}

  The connection of these concepts to linearly closed clonoids on $\Z_p$
  will be given in the next subsection.
  
  Let $p$ be a prime. 
  Let $ f \in \Pol{\Z_p}$, and let $N:= \langle x_i^p - x_i \mid i\in \N \rangle$ be the ideal in $\Pol{\Z_p}$
  generated by $x_i^p -x_i$ for $i \in \N$.
  Following \cite[Chapter 15.3]{E:CA}
  there is a remainder $f'$ of $f$ with respect to $N$.
  In this case the remainder is unique.
  We denote this remainder by $\pRep{f}$ and call $\pRep{f}$ the \emph{$p$-representative} of $f$.
  The $p$-representative has the property that all exponents of the variables of $\pRep{f}$ are less or equal $p-1$.
  We denote the set of all polynomials on $\Z_p$ with this property
  by $\redPol{\Z_p}:=  \{0\} \cup 
  \{\sum_{i=1}^n a_i \prod_{j=1}^m x_j^{\alpha(i,j)} \mid n, m \in \N, \forall i \leq n\colon a_i \in \Z_p \setminus \{0\}, 
  \; \forall i \leq n \forall j \leq m \colon \alpha(i,j) \leq p-1, \; \boldsymbol{\alpha}(i) \neq \boldsymbol{\alpha}(i') \textmd{ for } i \neq i'\}$.
  We call $\redPol{\Z_p}$ the set of \emph{reduced polynomials over $\Z_p$}.
  
  \subsection{Induced Functions}
  
  Let $p$ be a prime, let $f \in \Pol{\Z_p}$ and let
   $n \in \N$ with $n \geq \maxInd{f}$.
  We denote the function from $\Z_p^n$ to $\Z_p$ which is induced by $f \in \Pol{\Z_p}$ on $\Z_p$, by
 $f^{\Z_p^n}$.  
  Since $a^p \equiv_p a$ for all $ a \in \Z_p$,
  we have 
  $f^{\Z_p^n} = (\pRep{f})^{\Z_p^n}$.
  Therefore we can always assume that a function of a linearly closed clonoid on $\Z_p$
  is induced by a polynomial from $\redPol{\Z_p}$.
  
  We mention that $\max(\{i \mid x_i \in \occVar{ \pRep{f}} \})$ can be smaller than $\maxInd{f}$,
  e.g. $p = 3$, $f = x_1 + x_2 + 2x_2^3$ then $\pRep{f} = x_1$.
  This does not cause any problems: 
  Let $C$ be a linearly closed clonoid on $\Z_p$.
  If $f^{\Z_p^n}\in C$, then $\pRep{f}^{\Z_p^{m}} \in C$ for all $m \geq \max(1, \maxInd{\pRep{f}})$
  since $C$ is a linearly closed clonoid.
  
  Now we are able to give the connection between polynomial linearly closed clonoids and linearly closed clonoids on $\Z_p$.
  
  \begin{lem}\label{lem:polToFun}
   Let $C \subseteq \Pol{\Z_p}$ be a polynomial linearly closed clonoid.
   Then the set $ \{ \pRep{f}^{\Z_p^n} \mid f \in C, n \geq \max(1,\maxInd{\pRep{f}}) \}$
   is a linearly closed clonoid on $\Z_p$.
  \end{lem}
  
  \begin{proof}
   Immediate.
  \end{proof}

  \begin{pro} \label{pro:genCCL} 
   Let $C, D$ be linearly closed clonoids on $\Z_p$ and let $F, G \subseteq \Pol{\Z_p}$ such that
   $C = \genLC{\{f^{\Z_p^n} \mid f \in F, n = \max(1, \maxInd{f})\}}$
   and $D =  \genLC{\{g^{\Z_p^n} \mid g \in G, n = \max(1, \maxInd{G})\}}$.
    If $\genPLC{F} = \genPLC{G}$, then $C = D$.
  \end{pro}

  \begin{proof}
   Let $$C' := \{ \pRep{f}^{\Z_p^n}  \mid f \in \genPLC{F}, n \geq \max(1,\maxInd{\pRep{f}}) \}$$
   and $$D' := \{ \pRep{f}^{\Z_p^n}  \mid f \in \genPLC{G}, n \geq \max(1,\maxInd{\pRep{f}}) \}.$$
   First we show that 
   $C = C'$ and $D = D'$.
   $C'$ is by Lemma \ref{lem:polToFun} a linearly closed clonoid on $\Z_p$
   and it contains all generators of $C$ since $F \subseteq \genPLC{F}$.
   Hence $C \subseteq C'$.
   For the other inclusion, let $f \in C'$.
   Then there exists $g \in \genPLC{F}$ such that $f = \pRep{g}^{\Z_p^n}$ for some $n \in \N$.
   We have $\pRep{g}^{\Z_p^n} \in C$ since $C = \genLC{\{f^{\Z_p^n} \mid f \in F, n = \max(1, \maxInd{f})\}}$.
   Hence $f\in C$ and 
   thus $C = C'$.   
   For $D = D'$ the proof is analogous.
   
  Now we get if $\genPLC{F} = \genPLC{G}$, then  $ C = C' = D' = D$.
  \end{proof}

  In order to study linearly closed clonoids on $\Z_p$ we will investigate
  some properties of polynomial linearly closed clonoids and
  by Lemma \ref{lem:polToFun} and Proposition \ref{pro:genCCL} we are then able to transfer these results to linearly closed clonoids on $\Z_p$.

\section{$p$-minor subsets} \label{sec:pminSub}

Let $p$ be a prime, fixed for the rest of the section. 
In this section we will get the main result of the paper by using
$p$-minor subsets.
Remember, we call $M \subseteq N_0$ a $p$-minor subset if $n-(p-1) \in M$ for all $n \in M$ with $n > p-1$. 
Here a short example:
 
 \begin{exa}
   Let $p = 3$ and $M$ be a $3$-minor subset.
 If $9 \in M$, then $\{7, 5, 3, 1\} \subseteq M$.
 If also $4 \in M$, then $\{7,5,4,3,2,1\} \subseteq M$, but not necessarily $0 \in M$.
 \end{exa}

 The next definition will connect $p$-minor subsets to linearly closed clonoids on $\Z_p$.

 \begin{de}
  For every $p$-minor subset $M$ we define 
   \begin{align*}
    \mathcal{C}(M) := & \{ 0^{\Z_p^n} \mid n \in \N \} \\ &\cup \{  f^{\Z_p^n} \mid f \in \redPol{\Z_p}, \; n \geq \max(1, \maxInd{f}), \; \setTotDeg{f} \subseteq M  \}
   \end{align*}
  and 
  $$\mathcal{P}(M) := \{ 0 \} \cup \{  f \mid f \in \Pol{\Z_p}, \; \setTotDeg{f} \subseteq M  \}.$$
   \end{de}
 
 The main goal is to show that 
 $\mathcal{C}$ is a lattice isomorphism from 
 the set of all $p$-minor subsets to the set of all linearly closed clonoids on $\Z_p$.
 In order to show this, we will work on polynomials and
 use the connection between $\mathcal{C}$ and $\mathcal{P}$:
  
  \begin{lem}\label{lem:CconP}
   Let $M$ be a $p$-minor subset.
   Then 
   $$\mathcal{C}(M) = \{0^{\Z_p^n} \mid n \in \N \}\cup \{\pRep{f}^{\Z_p^n}\mid f \in \mathcal{P}(M), n \geq \max(1, \maxInd{\pRep{f}}) \}.$$
  \end{lem}
  
  \begin{proof}
   If $\setTotDeg{f} \subseteq M$, then $\setTotDeg{\pRep{f}} \subseteq M$,
   since for $n\in M$ we have $n-t(p-1) \in M$ for all $t \in \N_0$ with $n-t(p-1) > 0$.
   Hence $\pRep{f} \in \mathcal{P}(M)$ for all $f\in \mathcal{P}(M)$
   and thus we get 
   \begin{align*}
   \{\pRep{f}^{\Z_p^n}\mid f \in \mathcal{P}(M), n \geq \max(1, \maxInd{\pRep{f}}) \} \\= \{f^{\Z_p^n}\mid f \in \mathcal{P}(M), n \geq \max(1, \maxInd{f}) \} .
   \end{align*}
   Now the result follows
   since 
    $$\mathcal{C}(M) = \{ 0^{\Z_p^n} \mid n \in \N \}  \cup \{  f^{\Z_p^n} \mid f \in \redPol{\Z_p} \cap \mathcal{P}(M), \; n \geq \max(1, \maxInd{f}) \}.$$
  \end{proof}

 The first step is to verify that the function $\mathcal{C}$ is a map into the set of all linearly closed clonoids on $\Z_p$.

  \begin{pro}\label{pro:MsubIsLCC}
 Let $M$ be a $p$-minor subset.
 Then $\mathcal{P}(M)$ is a polynomial linearly closed clonoid,
 and $\mathcal{C}(M)$ is a linearly closed clonoid on $\Z_p$.
 \end{pro}
 
 \begin{proof}
  If $M = \emptyset$, then $\mathcal{P}(M) = \{0 \}$,
  which is a  polynomial linearly closed clonoid.
  Now let $M \neq \emptyset$. 
  First we show $L(\mathcal{P}(M)) \subseteq \mathcal{P}(M)$.
  Let $n \in \N_0$ and $l \in \linPol^{\ar{n}}$.
  If  $\setTotDeg{l} =\emptyset$, then $l=0$.
  We know that $0 \in \mathcal{P}(M)$.
  Let $\setTotDeg{l} =\{1\}$.
  Now let $m \in \N_0$ and let $f_1, \ldots, f_n\in \mathcal{P}(M)^{\ar{m}}$.
  We define $f := l(f_1, \ldots, f_n)$.
  Since  $\setTotDeg{f_i} \subseteq M$ for all $i\leq n$ and $\setTotDeg{l} =\{1\}$, we have $\setTotDeg{f} \subseteq M$.
  Hence, $f \in \mathcal{P}(M)$.
  Now we show $(\mathcal{P}(M))L \subseteq \mathcal{P}(M)$.
  Let $n \in \N_0$ and let $f \in \mathcal{P}(M)^{\ar{n}}$.
  Then let $m \in \N_0$ and let $l_1, \ldots, l_n \in \linPol^{\ar{m}}$.
  We define $g := f(l_1, \ldots, l_n)$.
  Now we have $\setTotDeg{g} \subseteq \setTotDeg{f}$
  and since $\setTotDeg{f} \subseteq M$ we have $g \in \mathcal{P}(M)$.
  Therefore $\mathcal{P}(M)$ is a polynomial linearly closed clonoid,
  and by Lemma \ref{lem:polToFun} and Lemma \ref{lem:CconP}, $\mathcal{C}(M)$ is a linearly closed clonoid on $\Z_p$.
 \end{proof} 
 
 \begin{lem}\label{lem:Chom}
  $\mathcal{C}$ is a lattice homomorphism from 
  the set of all $p$-minor subsets to the set of all linearly closed clonoids on $\Z_p$.
 \end{lem}
 
 \begin{proof}
  Let $M_1$ and $M_2$ be two $p$-minor subsets.
  Then $M_1 \cup M_2$ and $M_1 \cap M_2$ are $p$-minor subsets.
  It holds
  that $\mathcal{C}(M_1 \cup M_2) = \mathcal{C}(M_1) \cup \mathcal{C}(M_2)$
  and  $\mathcal{C}(M_1 \cap M_2) = \mathcal{C}(M_1) \cap \mathcal{C}(M_2)$.
 \end{proof}
 
 The next lemma shows that $\mathcal{C}$ is a function which preserves the ordering. 
 
    \begin{lem} \label{lem:COrd}
  Let $M_1, M_2$ be two $p$-minor subsets.
  Then $M_1 \subseteq M_2$ if and only if $\mathcal{C}(M_1) \subseteq \mathcal{C}(M_2)$.  
  \end{lem}
  
  \begin{proof}
   ``$\Rightarrow$'': 
   If $M_1 \subseteq M_2$, then $\mathcal{C}(M_1) \subseteq \mathcal{C}(M_2)$,
   by the definition of $\mathcal{C}$.
   
   ``$\Leftarrow$'':   
   If $M_1 = \emptyset$, this is obvious.
  Let $m \in M_1$.
  If $m = 0$, then 
   $ c^{\Z_p^n} \in \mathcal{C}(M_1)$ for all $c\in \Z_p$ and for all $n \in \N$.
   Let $c\in \Z_p$ and $n\in \N$.
    Since $\mathcal{C}(M_1) \subseteq \mathcal{C}(M_2)$,
  we have $ c^{\Z_p^n} \in \mathcal{C}(M_2)$.
  The $p$-representative of a constant $c$ is $c$, hence $m \in M_2$.
  
  If $m > 0$, we let $n\geq m$.
  Then $(x_1\cdots x_m)^{\Z_p^n} \in \mathcal{C}(M_1)$.
  Since $\mathcal{C}(M_1) \subseteq \mathcal{C}(M_2)$,
  we have $(x_1\cdots x_m)^{\Z_p^n} \in \mathcal{C}(M_2)$.
  Since $x_1\cdots x_m$ is the $p$-representative for a polynomial which induces $(x_1\cdots x_m)^{\Z_p^m}$
  it follows that $m \in M_2$.
  \end{proof}

  We now get as an easy consequence from Lemma \ref{lem:COrd} that $\mathcal{C}$ is an injective function. 
  
 \begin{lem}\label{lem:CisInj}
  Let $M_1, M_2$ be two $p$-minor subsets.
  If  $\mathcal{C}(M_1) = \mathcal{C}(M_2)$, then $M_1 = M_2$.
 \end{lem}
 
 \begin{proof}
  Follows from Lemma \ref{lem:COrd}. 
 \end{proof}

 The next step is to show that $\mathcal{C}$ is surjective.
 To this end, we will now prove some facts about polynomial linearly closed clonoids.

   \begin{lem}\label{lem:help1RedLemPol}
   Let $d \in \N$ and 
   let $g \in \Pol{\Z_p}$ with $\max(\setTotDeg{g}) \leq d$, and
   $\mathbf{1}_d \not \in \setDeg{g}$.
   Then $x_1\cdots x_d \in \genPLC{\{x_1 \cdots x_d + g\}}$.
  \end{lem}
  
  \begin{proof}
  If $g=0$, the result is obvious.
  Let $g \neq 0$ and let $C := \genPLC{\{x_1 \cdots x_d + g\}}$.
  First, we show that 
  there exists a $g' \in \Pol{\Z_p}$ with $\max(\setTotDeg{g'}) \leq d$, 
   $\mathbf{1}_d \not \in \setDeg{g'}$ and $\occVar{g'} \subseteq \{x_1, \ldots, x_d\}$
  such that $x_1\cdots x_d + g' \in  C$.
  If $\occVar{g} \subseteq \{x_1, \ldots, x_d\}$, we are done.
  Otherwise
  let $B:= \occVar{g} \setminus\{x_1, \ldots, x_d\}$.
  Let $k \in \N$ and let $b_1, \ldots, b_k \in \N$ such that $b_1< \ldots < b_k$ 
  and $B = \{x_{b_1}, \ldots, x_{b_k}\}$.
   We set $g' := g \circ_{(b_1, \ldots, b_k)} 0$.
   We have $\occVar{g'} \subseteq \{x_1, \ldots, x_d\}$
    and thus
   $$ x_1\cdots x_d +g' = (x_1\cdots x_d + g)\circ_{(b_1, \ldots, b_k)} 0 \in  C. $$
   
   Now let $g' \in \Pol{\Z_p}$ with $\max(\setTotDeg{g'}) \leq d$, 
   $\mathbf{1}_d \not \in \setDeg{g'}$ and $\occVar{g'} \subseteq \{x_1, \ldots, x_d\}$
  such that  $x_1\cdots x_d + g' \in  C$.
  
  We proceed by induction on the number of monomials of $g'$ to show
  that $x_1\cdots x_d \in \genPLC{ \{x_1\cdots x_d + g'\}}\subseteq C$.
  If the number of monomials of $g'$ is zero we are done.
   For the induction step, we observe that there exists $l \leq d$ and a monomial $m$ of $g'$ such that $x_l$ does not appear in $m$.
   Then we get
    \begin{align*}
      (x_1\cdots x_d + g') - ((x_1\cdots x_d + g')\circ_{(l)} 0)=  x_1 \cdots x_d + g'' \in  C,
    \end{align*}
   where $g'' := g' -g'\circ_{(l)} 0$
   has the properties that
   $\max(\setTotDeg{g''}) \leq d$, 
   $\mathbf{1}_d \not \in \setDeg{g''}$, $\occVar{g''} \subseteq \{x_1, \ldots, x_d\}$
   and $g''$ has fewer monomials than $g'$ since the monomial $m$ is cancelled by the calculation $g' -g'\circ_{(l)} 0$.   
   Hence our induction hypothesis yields $x_1 \cdots x_d \in  C$.
  \end{proof}

    \begin{lem} \label{lem:redLemPol}
   Let $0 \neq f \in \redPol{\Z_p}$ with $d:=\max(\setTotDeg{f})$.
   Then the following hold:
   \begin{enumerate}
     \item If $d=0$, then $1 \in \genPLC{\{f\}}$.
    \item If $d>0$, then $x_1 \cdots x_d \in \genPLC{\{f\}}$.
   \end{enumerate}
  \end{lem}
  
  \begin{proof}
  Let $C := \genPLC{\{f\}}$.
  We write $f$ in the form $f = \sum_{i=1}^n a_i \prod_{j = 1}^m x_j^{\alpha(i,j)}$
  such that 
  $a_i \neq 0$ for all $i \leq n$,
  and $\boldsymbol{\alpha}(i) \neq \boldsymbol{\alpha}(i')$ for all $i, i' \leq n$ with $i \neq i'$.
 
   If $d = 0$, then $f= a$ for some $a \in \Z_p\setminus{\{0\}}$.
   There exists a $b \in \Z_p$ such that $b \cdot a \equiv_p 1$,
   hence we have $1 \in C$.
   
  If $d > 0$, we fix an $i'\leq n$ with $\sum_{j=1}^m \alpha(i',j) = d$.
   Now we show for all $m \in \N$ with $\size{\occVar{a_{i'} \mathbf{x}^{\boldsymbol{\alpha}(i')}}} \leq m \leq d$
   that there exists a monomial $h := a\mathbf{x}^{\boldsymbol{\beta}}$ in $f$ with $a\in \Z_p\setminus\{0\}$ with  $\lvert \occVar{h} \rvert = m$,
   $\totD{h} = d$ and $\beta(j) \leq p-1$ for all $j \leq m$ 
   and there exists $f' \in \Pol{\Z_p}$ with $\max(\setTotDeg{f'}) \leq d$ and $\boldsymbol{\beta} \not \in \setDeg{f'}$ 
   such that $h + f' \in C$. 
   
   For the induction base $m = \lvert \occVar{\boldsymbol{\alpha}(i')} \rvert$, we see that
   $ h = a_{i'} \mathbf{x}^{\boldsymbol{\alpha}(i')}$ and $f'  = \sum_{\substack{i=1\\i\neq i'}}^n a_i \prod_{j = 1}^m x_j^{\alpha(i,j)}$.
   
   For the induction step, let $m < d$ such that the induction hypothesis holds for $m$.
   By the induction hypothesis there 
   is a monomial $h = a\mathbf{x}^{\boldsymbol{\beta}}$ and $f' \in \Pol{\Z_p}$ with 
   $a\in \Z_p\setminus\{0\}$ with  $\lvert \occVar{h} \rvert = m$,
   $\totD{h} = d$ and $\beta(j) \leq p-1$ for all $j \leq m$, and
   $\max(\setTotDeg{f'}) \leq d$ and $\boldsymbol{\beta} \not \in \setDeg{f'}$ 
   such that
   $h + f' \in C$. 
   Since $m < d$,
   we have $l \in \N$ such that $\beta(l) > 1$.
   We define $f'' := f'  \circ_{(l)} (x_{l} + x_{m+1})$.
         Now we calculate
    \begin{align*}
         (h + f' ) \circ_{l} (x_{l} + x_{m+1}) &=
          a\mathbf{x}^{\boldsymbol{\beta}}  \circ_{(l)} (x_{l} + x_{m+1}) +  f'  \circ_{(l)} (x_{l} + x_{m+1}) \\ & =
          a(x_{l} + x_{m+1})^{\beta(l)} \cdot \prod_{\substack{j =1 \\ j\neq l}}^m x_j^{\beta(j)}  + f'' \\ &=
          a(\sum_{k=0}^{\beta(l)} \binom{\beta(l)}{k} x_{l}^{\beta(l)-k}x_{m+1}^k  ) \cdot \prod_{\substack{j =1  \\ j\neq l}}^m x_j^{\beta(j)}  + f''\\ &=
          \beta(l)\cdot a \cdot \mathbf{x}^{(\beta(1), \ldots, \beta(l-1), \beta(l)-1, \beta(l+1), \ldots, \beta(m), 1)} \\& \;\;\;+ 
          (\sum_{\substack{k=0\\k\neq 1}}^{\beta(l)} \binom{\beta(l)}{k} x_{l}^{\beta(l)-k}x_{m+1}^k  ) \cdot \prod_{\substack{j =1  \\ j\neq l}}^m x_j^{\beta(j)} + f''.
    \end{align*}
   
   Since $\beta(l) \leq p - 1$,
   we have $\beta(l)\cdot a \not\equiv_p 0$.
   
   Let $$h' := \beta(l)\cdot a \cdot \mathbf{x}^{(\beta(1), \ldots, \beta(l-1), \beta(l)-1, \beta(l+1), \ldots, \beta(m), 1)} $$
   and $$g:=  (\sum_{\substack{k=0\\k\neq 1}}^{\beta(l)} \binom{\beta(l)}{k} x_{l}^{\beta(l)-k}x_{m+1}^k  ) \cdot \prod_{\substack{j =1  \\ j\neq l}}^m x_j^{\beta(j)} + f'' .$$
   
   Then $h'$ also has the properties that
   $\totD{h} = d$, $\beta(j) \leq p-1$ for all $j \leq m + 1$ and $\lvert \occVar{h} \rvert = m + 1$.
   The polynomial $g$ has the properties that
   $\max(\setTotDeg{g}) \leq d$ and
   $(\beta(1), \ldots, \beta(l-1), \beta(l)-1, \beta(l+1), \ldots, \beta(m), 1) \not \in \setDeg{g} $
   since $k \neq 1$ and $f'' = f'  \circ_{(l)} (x_{l} + x_{m+1})$ where $\boldsymbol{\beta} \not \in \setDeg{f'}$.   
   Since $(h + f' ) \circ_{l} (x_{l} + x_{m+1}) \in C$,
   $h'$ and $f''':= g$ are the searched polynomials.
   This concludes the induction step.

   Hence
   there exist $a\in\Z_p\setminus\{0\}$, $h:= x^{\boldsymbol{\beta}} \in \Pol{\Z_p}$ with $\lvert \occVar{h} \rvert = d$ and $\totD{h}=d$, 
   $f' \in \Pol{\Z_p}$ with $\max(\setTotDeg{f'}) \leq d$ and $\boldsymbol{\beta} \not \in \setDeg{f'}$
   such that
   $$a\cdot h +f' \in C. $$
   
  There exists $b \in \Z_p$ such that $ b\cdot a \equiv_p 1$.
  Thus,
  $$h + b\cdot f' \in C. $$
  
  Since $C$ is closed under composition with linear polynomials, it is also closed under relabeling the variables.
  After relabeling the variables we get
  $$x_1 \cdots x_d + f'' \in C, $$
  for some $f'' \in \Pol{\Z_p}$.
  Now Lemma \ref{lem:help1RedLemPol} yields that $x_1 \cdots x_d \in C$.
  \end{proof}

  With the help of Lemma \ref{lem:redLemPol} we are now able to show the 
  next crucial lemma by using the connection between 
  linearly closed clonoids on $\Z_p$ and polynomial linearly closed clonoids.
 
  \begin{lem} \label{lem:redLCC}
  Let $C$ be a linearly closed clonoid on $\Z_p$ and
  let $P := \{x_1\cdots x_i \mid i\in \N\} \cup \{1\}$.
  Then there exists $T \subseteq P$ such that 
  $C = \genLC{\{f^{\Z_p^n}\mid  f \in T,\; n = \max(\maxInd{f}, 1) \}}$.
 \end{lem}

 \begin{proof}
 If $C = \{0^{\Z_p^i} \mid i \in \N \}$, let $T = \{\}$.
 Let $C \neq  \{0^{\Z_p^i} \mid i \in \N \}$.
    Let $G \subseteq \Pol{\Z_p}$
  such that $$C = \genLC{\{g^{\Z_p^n} \mid g \in G, \; n = \max(\maxInd{g} ,1)\}}.$$
   We know that every function $f^{\Z_p^n}$ of $C$ can be induced by 
 the $p$-representative $\pRep{f}$ of $f$.
  Hence we can assume $g \in \redPol{\Z_p}$ for all $g \in G$.
 
   For each $i > 0$ let $f_i := x_1 \cdots x_i$ 
 and $f_0 := 1$.
 Let
 $I := \{i \in \N_0 \mid \exists g \in G\colon i \in \setTotDeg{g}\} $
 and let 
 $ T = \{f_i\mid i \in I \}. $
 We want to show that 
 \begin{equation}\label{eq:redLCC1}
  C = \genLC{\{f^{\Z_p^n}\mid  f \in T,\; n = \max( \maxInd{f}, 1)\}}.
 \end{equation}

 By Proposition \ref{pro:genCCL} it is sufficient to show that
 \begin{equation}\label{eq:redLCC2}
 \genPLC{G} = \genPLC{T}.
 \end{equation}
 
 ``$\subseteq$'':
  It is sufficient to show that $g \in \genPLC{T}$ for $g \in G$.
   Let $g \in G$.
  If $g = 0$, we are done since $0 \in\genPLC{T}$.
  If $g \neq 0$, then let
  $g = \sum_{i=1}^n a_i \prod_{j=1}^m x_j^{\alpha(i,j)}$.
  We define $d(l):=\sum_{j=1}^m \alpha(l,j) \in \setTotDeg{g}$ for all $l \leq n$.
  If $d(l) = 0$, then $f_0 \in T$, and therefore
  $a_l \in \genPLC{T}$.
  Now let $d(l) > 0$.
  Then $f_{d(l)} = x_1 \cdots x_{d(l)}$ and $f_{d(l)} \in T$.
  For each $l\leq n$, we define
  $$f'_l := \prod_{j=1}^m x_j^{\alpha(l,j)} = f_{d(l)}(\underbrace{x_1, \ldots, x_1}_{\alpha(l, 1)}, \ldots, \underbrace{x_m, \ldots, x_m}_{\alpha(l,m)})$$
  where for all $j \leq m$, $x_j$ appears exactly $\alpha(l,j)$ times.
  By the axioms of a polynomial linearly closed clonoid 
  we now have $f'_l \in \genPLC{T}$ for all $l\leq n$,  and thus
  also
  $g = \sum_{i=1}^n a_i f'_i \in \genPLC{T}$.
  
    ``$\supseteq$'': 
  It is sufficient to show that $f_i \in \genPLC{G}$ for
  $i \in I$.
  Let $i' \in I$.
  Then there is a $g \in G$ such that $i' \in \setTotDeg{g}$.
  Let $g = \sum_{i=1}^n a_i \prod_{j=1}^m x_j^{\alpha(i,j)}$.
  We proceed by induction on 
  the total degrees of $g$
  to show 
  $f_{d} \in  \genPLC{\{g\}}$ for $d \in \setTotDeg{g}$.
  
  For the induction base let $d = \max(\setTotDeg{g})$. 
  Since $g \in \redPol{\Z_p}$, Lemma \ref{lem:redLemPol} yields that 
  $f_{d} \in \genPLC{\{g\}}$.
  For the induction step let $d < \max(\setTotDeg{g})$.  
  Then for all $d' \in \setTotDeg{g}$ with $d' > d$  the induction hypothesis yields
  that $f_{d'}  \in \genPLC{\{g\}}$.
  
  This means for all $i\leq n$ with $\sum_{j=1}^m \alpha(i, j) > d$
  we have $\prod_{j=1}^m  x_j^{\alpha(i,j)} \in C$ and thus 
  for all $i \leq n$ with $\sum_{j=1}^m \alpha(i, j) > d$ we have $a_i\prod_{j=1}^m  x_j^{\alpha(i,j)} \in \genPLC{g}$.
  Let $D := \{i\leq n \mid \sum_{j=1}^m \alpha(i, j) > d \}$.
  
  Now we calculate
  $$ g' = g - (\sum_{i \in D} a_i \prod_{j=1}^m  x_j^{\alpha(i,j)}) \in \genPLC{g}.$$
  Then $\max(\setTotDeg{g'})$ is equal to $d$. 
  Now Lemma \ref{lem:redLemPol}  yields that 
  $f_{d}$ is an element of $\genPLC{\{g\}}$, which concludes the induction step.
  Since $i' \in \setTotDeg{g}$, we have $f_{i'} \in \genPLC{\{g\}}$
  and thus $f_{i'} \in \genPLC{G}$.
  This concludes the proof of \eqref{eq:redLCC2}
  and hence \eqref{eq:redLCC1} holds.
 \end{proof}

 Now we are able to show that $\mathcal{C}$ is surjective.

  \begin{lem}\label{lem:CisSurj}
  Let $C$ be a 
  linearly closed clonoid on $\Z_p$. 
  Then there exists a $p$-minor subset $M$ such that
  $C = \mathcal{C}(M)$.
 \end{lem}
 
 \begin{proof} 
 If $C = \{0^{\Z_p^i} \mid i \in \N\}$, let $M := \{\}$.
 Let $C \neq  \{0^{\Z_p^i} \mid i \in \N\}$.
  Let $f_0:= 1$ and  $f_i := x_1 \cdots x_i$ for $i \in \N$.
 By Lemma \ref{lem:redLCC} there exists $I\subseteq \N_0$
 such that $\genLC{\{f_i^{\Z_p^i} \mid i\in I\}} = C$.
 Now let $M$ be the smallest $p$-minor subset which contains $I$.
  Now we show that
  \begin{equation}\label{eq:CisSurj1}
  C = \mathcal{C}(M).
  \end{equation}

 ``$\subseteq$'': Holds since $T \subseteq \mathcal{P}(M)$.
  
  ``$\supseteq$'':
  We have
  $\mathcal{C}(M) = \genLC{\{f_i^{\Z_p^i} \mid i \in M \}}$, and hence 
  it is sufficient to show for all $i \in M$ that $f_i^{\Z_p^i}$ lies in $C$.
  Let $n \in M$.
  If $n= 0$, then $0 \in M$ and $f_0^{\Z_p^1} \in C$.
  Let $n > 0$.
  By the definition of $M$ there is a $t \in \N_0$
  such that $(f_{n + t(p-1)}) \in T$.
  Now we calculate 
  \begin{equation*}
   \begin{aligned}
      f_{n + t(p-1)}  \circ_{(n, \ldots, n+t(p-1))}(x_{n}) & = \prod_{i=1}^{n-1}x_i \cdot x_{n}^{t(p-1)+1} \\
  & \equiv_p \prod_{i=1}^{n} x_i.
   \end{aligned}
  \end{equation*}
  Hence
  $f_n^{\Z_p^n} \in C$.
  This concludes the proof of \eqref{eq:CisSurj1}.
 \end{proof}

 By Lemma \ref{lem:Chom}, Lemma \ref{lem:CisInj} and Lemma \ref{lem:CisSurj} 
 we have proven the following theorem, which proves Theorem \ref{thm:ordIso}:
 
 \begin{thm}\label{thm:CisIso}
 $\mathcal{C}$ is a lattice isomorphism from 
 the set of all $p$-minor subsets of $\N_0$ to the set of all linearly closed clonoids on $\Z_p$.
 \end{thm}
 
 In the next section we will see some consequences of Theorem \ref{thm:CisIso}.
  
\section{Results and consequences}\label{sec:resAndCons}
 
  Before we come to linearly closed iterative algebras and to clones which contain $+$
  we see that for a prime $p \neq 2$ there are infinitely many different non finitely generated $p$-linear closed
  clonoids.
 
 \begin{pro} \label{pro:infNonFinGen}
 There are 2 different non finitely generated linearly closed clonoids on $\Z_2$.
 For $p$ an odd prime, there are countably infinitely many different non finitely generated linearly closed clonoids on $\Z_p$.
 \end{pro}

 \begin{proof}
  Let $C$ be a linearly closed clonoid on $\Z_p$.
  $M$ is finite if and only if $\mathcal{C}(M)$ is finitely generated,
  since a finite $M$ gives a bound on the total degrees of the polynomials which induces the generators for $\mathcal{C}(M)$
  and if $M$ is infinite there is no maximum on the total degrees.
  
  If $p=2$, then $\mathcal{C}(\N)$ and $\mathcal{C}(\N_0)$ are different and not finitely generated.
  These are the only cases 
  where $\mathcal{C}(M)$ is not finitely generated, since for all $2$-minor subsets with
  $M\neq \N$ and $M\neq \N_0$, $M$ is finite.
  
  Let $p>2$.
  Then there are $m_1, m_2 \leq p-1$ with $0 < m_1, m_2$ such that $m_1 \neq m_2$.
  We define 
  $M(n) := \{m_1 +t_1(p-1) \mid t_1 \in \N_0\} \cup \{ m_2+t_2(p-1) \mid t_2 \leq n\}$
  for $n \in \N_0$.
  Then $M(n)$ is an infinite $p$-minor subset for all $n \in \N_0$.
  This means there are infinitely many different $p$-minor subsets,
  and therefore infinitely many non finitely generated linearly closed clonoids on $\Z_p$.
  Furthermore, we see that there are only countably many linearly closed clonoids on $\Z_p$
  since $M$ is completely determined by the family
  $\langle \sup(M \cap [i]_{p-1})\mid i \in \{1, \ldots, p-1\} \rangle $
  in $(\N_0 \cup \{\infty\})^{p-1}$ and whether $0\in M$. 
\end{proof}

Figure~\ref{fig:lat3minS} shows the lattice of $3$-minor subsets without $0$.
By Theorem \ref{thm:CisIso} this lattice is isomorphic
to the lattice of linearly closed clonoids on $\Z_3$ which do not contain the functions induced by constants (except the constant $0$).

\begin{figure}[ht!]
\begin{tikzpicture}
  \node (min) at (0,0) {$\emptyset$};
  \node (a) at (-1,1) {$\{1\}$};
  \node (b) at (1,1) {$\{2\}$};
  \node (c) at (-2,2) {$\{1, 3\}$};
  \node (d) at (0,2) {$\{1, 2\}$};  
  \node (e) at (2,2){$\{2,4\}$};
  \node (f) at (-3,3){$\{1,3,5\}$};
  \node (g) at (-1,3){$\{1,2,3\}$};
  \node (h) at (1,3){$\{1,2,4\}$};
  \node (i) at (3,3){$\{2,4,6\}$};
  \node (j) at (-4,4){$\cdot$};
  \node (k) at (-2,4){$\cdot$};
  \node (l) at (0,4){$\cdot$};  
  \node (m) at (2,4){$\cdot$}; 
  \node (n) at (4,4){$\cdot$}; 
  \node (o) at (-4,4.2){$\cdot$};
  \node (p) at (-2,4.2){$\cdot$};
  \node (q) at (0,4.2){$\cdot$};  
  \node (r) at (2,4.2){$\cdot$}; 
  \node (s) at (4,4.2){$\cdot$}; 
  \node (t) at (-4,4.4){$\cdot$};
  \node (u) at (-2,4.4){$\cdot$};
  \node (v) at (0,4.4){$\cdot$};  
  \node (w) at (2,4.4){$\cdot$}; 
  \node (x) at (4,4.4){$\cdot$};

  \draw (min) -- (a) -- (c) -- (f) -- (j)
  		(min) -- (b) -- (e) -- (i) -- (n)
		(a) -- (d) -- (g) --(l)
		(b) -- (d) -- (h) -- (m)
		(c) -- (g)
		(e) --(h)
		(f) -- (k)
		(i) -- (n)
		(h) -- (l)
		(g) -- (k)
		(i) -- (m);

\end{tikzpicture}
\caption{Lattice of those $3$-minor subsets that do not contain $0$} \label{fig:lat3minS}
\end{figure}
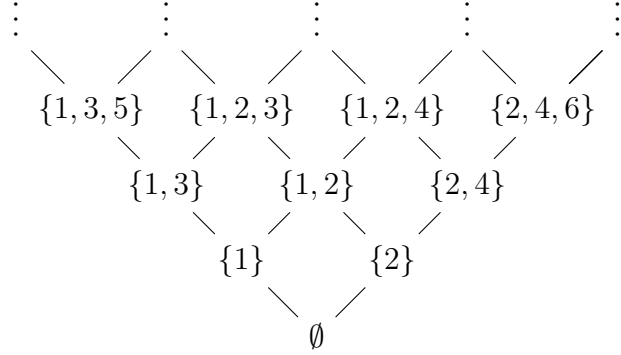

 Let $p$ be a prime.
 Now we use our results of linearly closed clonoids on $\Z_p$
 to characterize linearly closed iterative algebras on $\Z_p$.
 Remember, $C \subseteq \bigcup \{ \Z_p^{\Z_p^n} \mid n \in \N\}$
 is a linearly closed iterative algebra if $C$ is a linearly closed clonoid on $\Z_p$
 and $ C(C\cup J_{\Z_p} ) \subseteq C$.
 Since $J_{\Z_p} \subseteq \linFun{\Z_p}$ and by the Associativity Lemma of~\cite{CF:FC}
 this is equivalent to $ \linFun{\Z_p}(C(C\cup \linFun{\Z_p})) \subseteq C$.
 Let $C$ be a linearly closed iterative algebra on $\Z_p$.
 
 Theorem \ref{thm:charItAlg} now characterizes all linearly closed iterative algebras on
 $\Z_p$ for any prime $p$.
 
 \begin{thm}\label{thm:charItAlg}
  Let $p$ be a prime number and let $M$ be a $p$-minor subset.
  Then $\mathcal{C}(M)$ is a linearly closed iterative algebra on $\Z_p$ 
  if and only if $M \subseteq \{0,1\}$, $M=\N_0$ or 
  $M = \{1+t  m \mid t \in \N_0\}$ for some $m \mid p-1$.
 \end{thm}

 \begin{proof}
 Let $m \in \N$ such that $m \mid p-1$.
 Then clearly 
 $M= \{1+t \cdot m \mid t \in \N_0\}$ is a $p$-minor subset.
 
 ``$\Rightarrow$'':
 Let $M \not \subseteq \{0, 1\}$ and $M \neq \N_0$ be a $p$-minor subset 
 such that $\mathcal{C}(M)$ is a linearly closed iterative algebra on $\Z_p$.
 
 Let $m \in M$ with $m \leq p-1$, let $n \in M\setminus\{0\}$, and let $t\in \N_0$.
 We show
  \begin{equation}\label{eq:charItAlg1}
  n + t(m-1) \in M.
  \end{equation}
 We know that $(x_1^{m})^{\Z_p^1} \in \mathcal{C}(M)$.
 Then $(x_1\cdots x_n)^{\Z_p^n} \in \mathcal{C}(M)$.
 Since $\mathcal{C}(M)$ is closed under composition,
 we also have $((x_1\cdots x_n) \circ_1 x_1^{m})^{\Z_p^n}  = (x_1^{m}x_2\cdots x_n)^{\Z_p^n} \in \mathcal{C}(M)$.
 Since $m < p$, $x_1^mx_2\cdots x_n$ is a $p$-representative of a polynomial which induces $(x_1^{m}x_2\cdots x_n)^{\Z_p^n}$  
 and thus we have $n + m - 1 \in M$.
 Therefore $M$ is infinite and thus 
 $n + t(m-1) \in M.$
 This concludes the proof of \eqref{eq:charItAlg1}.
 Now let $n > \gcd(m - 1, p-1)$.
 We show 
  \begin{equation}\label{eq:charItAlg2}
 n -\gcd(m - 1, p-1) \in M.
\end{equation}
 By the definition of a $p$-minor subset,
 $n-t(p-1) \in M$ for all $t \in \N_0$ with $t < \lfloor \frac{n}{p-1} \rfloor$.
 With \eqref{eq:charItAlg1} we then get 
 for all $t_1 \in \N_0$ and for all $t_2 \in \N_0$ with $t_2 < \lfloor \frac{n+t_1(m-1)}{p-1} \rfloor$
 that $n  + t_1(m-1) - t_2(p-1) \in M$.
 Therefore we also get that $n -\gcd(m - 1,  p-1) \in M$, hence \eqref{eq:charItAlg2} holds.
 
  Let $m^*$ be the smallest element in $M\setminus\{0,1\}$.
  We assume that $m^* - 1$ does not divide $p-1$.
 In this case $p$ has to be greater $2$ and
 thus $m^* > 3$ and $1 \leq \gcd(m^*-1, p-1) < m^*- 1$.
  By \eqref{eq:charItAlg2} we have $1 < m^*-\gcd(m^*-1, p-1) \in M$.
 This contradicts the fact that $m^*$ is the smallest element in $M\setminus \{0, 1\}$.
  Thus we know that $m^* - 1$ divides $p-1$.
  Now we show that 
  \begin{equation}\label{eq:charItAlg3}
   M = \{1 + t(m^* -1) \mid t \in \N_0 \}.
  \end{equation}
  
 ``$\supseteq$'':
 Since  $m^* - 1 \mid p-1$,
 there exists a $t_1 \in \N$
 such that $t_1 \cdot (m^*-1) = p - 1$.
 Since $m^* \in M$, we know $m^* + (t_1-1) (m^* - 1) = m^* +t_1(m^*-1) -m^* + 1 = p \in M$ by \eqref{eq:charItAlg1}.
 Hence $p-(p-1) = 1 \in M$.
 Therefore $1+t\cdot (m^*-1) \in M$ for $t\in \N_0$.
 
 ``$\subseteq$'':
 Let $m \in M$ with $m \not \in \{1 + t(m^* -1) \mid t \in \N_0 \}$.
 Then there exists $u \in \N$ such that
 $m^*+(u-1)(m^*-1) < m < m^* +u(m^*-1)$.
 Then $ 1 < m-u(m^*-1) < m^* $,
 and by \eqref{eq:charItAlg1} $m - u(m^* - 1) \in M$
 which contradicts that $m^* = \min(M\setminus \{0, 1\})$.
 This concludes the proof of \eqref{eq:charItAlg3}.

 ``$\Leftarrow$'':
 If $M \subseteq \{0,1\}$ or $M= \N_0$,
 the result is obvious.
 Let $M = \{1+t\cdot m \mid t \in \N_0\}$, where $m \mid p - 1$.
 Let $k, l \in \N_0$, let $f \in \mathcal{P}(M)^{\ar{k}}$ and let $g_1, \ldots, g_{k} \in \mathcal{P}(M)^{\ar{l}}$.
 It is sufficient to show that $f':=f(g_1, \ldots, g_{k}) \in \mathcal{P}(M)$.
 To this end we show that $\setTotDeg{f'} \subseteq M$.
 Let $d \in \setTotDeg{f'}$.
 Then there is a $\boldsymbol{\alpha} = (\alpha(1), \ldots, \alpha(k)) \in \setDeg{f}$
 and for all $i \leq k$ there is a $\boldsymbol{\beta}_i = (\beta_i(1), \ldots, \beta_i(l)) \in \setDeg{g_i}$ 
 such that $d = \sum_{i=1}^{k} \alpha(i)( \sum_{j=1}^{l} \beta_i(j))$.
 Since $\sum_{i=1}^{k} \alpha(i) \equiv_{m} 1$ and
 $\sum_{j=1}^{l} \beta_i(j)\equiv_{m} 1$ for all  $i \leq k$,
 we have $$d = \sum_{i=1}^{k} \alpha(i)( \sum_{j=1}^{l} \beta_i(j)) \equiv_{m} \sum_{i=1}^{k} \alpha(i) \equiv_{m} 1.$$
 Hence $d \in M$.
 \end{proof}

  Now we know how many different linearly closed iterative algebras there exist.
 
 \begin{cor}
  Let $p$ be a prime. Then there exist
  exactly $\lvert \{t \in \N  \mid t \textmd{ divides } p-1 \}\rvert + 5$ 
  different linearly closed iterative algebras on $\Z_p$.
 \end{cor}
 
 \begin{proof}
  By Theorem \ref{thm:charItAlg} we get the following:
  For each subset of $\{0,1\}$, for $\N_0$
  and for each divisor of $p-1$ we get a 
  linearly closed iterative algebra via $\mathcal{C}$.
  All of them are different, and 
  there do not exist more linearly closed iterative algebras.
  Thus there are 
   exactly $\lvert \{t \in \N \mid t \textmd{ divides } p-1 \}\rvert + 5$ 
  different linearly closed iterative algebras on $\Z_p$.
 \end{proof}

  Let $C$ be a linearly closed iterative algebra on $\Z_p$ and let $M$ be a $p$-minor subset such that $C = \mathcal{C}(M)$.
  $C$ is a clone on $\Z_p$ which contains $+$
 if $C$ contains the projections.
 This holds if and only
 if $1 \in M$.
 
 By Theorem \ref{thm:charItAlg} all clones on $\Z_2$
 which contain $+$
 are given by $\mathcal{C}(M)$, where
 \begin{itemize}
  \item $M = \{1\}$ (linear functions), or
  \item $M = \{0, 1\}$ (affine functions), or
  \item $M = \N$ ($0$-symmetric functions), or
  \item $M = \N_0$ (all functions).
 \end{itemize}
 
 These clones can also be easily found in Post's Lattice~\cite{P:PostLat}.
 Theorem \ref{thm:charItAlg} now gives a characterization for any prime $p$
 for all clones on $\Z_p$ which contain $+$ and thus we also get 
 Theorem \ref{cor:resultNoClones}.

 Figure \ref{fig:LRep} visualizes Theorem \ref{thm:charItAlg}.
 In Figure \ref{fig:IA} the lattice of linearly closed iterative algebras on additive groups of prime order
 is given and in Figure \ref{fig:pM1} we see the corresponding isomorphic lattice of $p$-minor subsets. 
 $\mathbf{LL}$ and $\mathbf{ML}$ are sublattices of the corresponding lattices which are isomorphic to the divisor lattice of $p-1$.
 In the case of $\mathbf{ML}$ the maximum element is $\N$ and the smallest element is $M_{p-1}:= \{  1 + t(p-1) \mid t \in \N_0\}.$
 In the case of $\mathbf{LL}$ the maximum element is the set of all $0$-symmetric functions on $\Z_p$ denoted by $\mathbf{\mathbf{O}}_{\Z_p}$,
 and the smallest element is $\{ 0^{\Z_p^n} \mid n \in \N \} \cup  \{f^{\Z_p^n} \mid  f \in \redPol{\Z_p}, \; n \geq \maxInd{f}, \; \setTotDeg{f} \subseteq M_{p-1} \}.$
We define the following other sets in  \ref{fig:IA}:
$\Delta_{\Z_p}$ denotes the set of functions induced by $0$, and
$\nabla_{\Z_p}$ denotes the set of all functions on $\Z_p$.

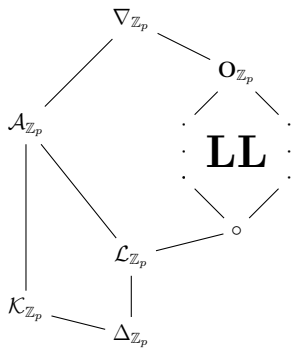
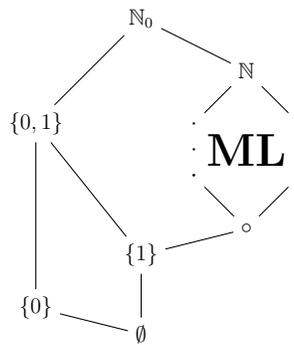
\begin{figure}[h!]
 \begin{minipage}[t]{.4\linewidth}
 \begin{center}
\begin{tikzpicture}[scale=0.7, transform shape]
  \node (max) at (0, 5.5) {$\nabla_{\Z_p}$};
  \node (a) at (-2,0) {$\constFun{\Z_p}$};
  \node (b) at (0,1) {$\linFun{\Z_p}$};
  \node (c) at (-2,3.5) {$\affFun{\Z_p}$};
  \node (h) at (2,1.5){$\circ$};
  \node (i) at (2, 4.5){$\mathbf{O}_{\Z_p}$};
  \node (j) at (1, 2.5){$\cdot$};
  \node (j1) at (1, 2.5){$\cdot$};
  \node (j2) at (1, 3){$\cdot$};
  \node (j3) at (1, 3.5){$\cdot$};
  \node (j5) at (2, 3){{\Huge $ \mathbf{LL}$}};
  \node (j7) at (3, 2.5){$\cdot$};
  \node (j8) at (3, 3){$\cdot$};
  \node (j9) at (3, 3.5){$\cdot$};
  \node (min) at (0,-0.5) {$\Delta_{\Z_p}$};

  \draw (min) -- (a) -- (c) 
  		(min) -- (b) -- (h)
  		(b) -- (c)
  		(h) -- (j1)
  		(h) -- (j7)
  		(j3) -- (i)
  		(j9) -- (i)--(max)
  		(c) -- (max);
 
\end{tikzpicture}\subcaption{Linearly closed iterative algebras on $\Z_p$}\label{fig:IA}
 \end{center}
 \end{minipage}
 \hfill
\begin{minipage}[t]{.4\linewidth}
\begin{center}
\begin{tikzpicture}[scale=0.7, transform shape]
  \node (max) at (0, 5.5) {$\N_0$};
  \node (a) at (-2,0) {$\{0\}$};
  \node (b) at (0,1) {$\{1\}$};
  \node (c) at (-2,3.5) {$\{0,1\}$};
  \node (h) at (2,1.5){$\circ$};
  \node (i) at (2, 4.5){$\N$};
  \node (j) at (1, 2.5){$\cdot$};
  \node (j1) at (1, 2.5){$\cdot$};
  \node (j2) at (1, 3){$\cdot$};
  \node (j3) at (1, 3.5){$\cdot$};
  \node (j5) at (2, 3){{\Huge $ \mathbf{ML}$}};
  \node (j7) at (3, 2.5){$\cdot$};
  \node (j8) at (3, 3){$\cdot$};
  \node (j9) at (3, 3.5){$\cdot$};
  \node (min) at (0,-0.5) {$\emptyset$};

  \draw (min) -- (a) -- (c) 
  		(min) -- (b) -- (h)
  		(b) -- (c)
  		(h) -- (j1)
  		(h) -- (j7)
  		(j3) -- (i)
  		(j9) -- (i)--(max)
  		(c) -- (max);
 
\end{tikzpicture}
 \subcaption{$p$-minor subsets}\label{fig:pM1}
 \end{center}
\end{minipage} 
\caption{Lattice of linearly closed iterative algebras}\label{fig:LRep}
\end{figure}

In Figure \ref{fig:clones} we see the lattice of clones on $\Z_p$
which contain $+$.

\begin{center}
\begin{figure}
 \begin{tikzpicture}[scale=0.7, transform shape]
  \node (max) at (0, 5.5) {$\nabla_{\Z_p}$};
  \node (b) at (0,1) {$\linFun{\Z_p}$};
  \node (c) at (-2,3.5) {$\affFun{\Z_p}$};
  \node (h) at (2,1.5){$\circ$};
  \node (i) at (2, 4.5){$\mathbf{O}_{\Z_p}$};
  \node (j) at (1, 2.5){$\cdot$};
  \node (j1) at (1, 2.5){$\cdot$};
  \node (j2) at (1, 3){$\cdot$};
  \node (j3) at (1, 3.5){$\cdot$};
  \node (j5) at (2, 3){{\Huge $ \mathbf{LL}$}};
  \node (j7) at (3, 2.5){$\cdot$};
  \node (j8) at (3, 3){$\cdot$};
  \node (j9) at (3, 3.5){$\cdot$};

  \draw  	(b) -- (h)
  		(b) -- (c)
  		(h) -- (j1)
  		(h) -- (j7)
  		(j3) -- (i)
  		(j9) -- (i)--(max)
  		(c) -- (max);
\end{tikzpicture}
\caption[LOF entry]{Clones with $+$ on groups of prime order} \label{fig:clones}
\end{figure}
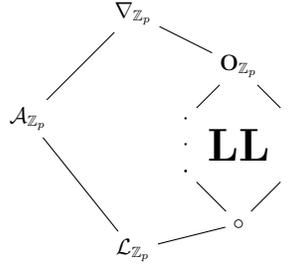
\end{center}

\section{Approach with Hagemann and Herrmann} \label{sec:HagHer}

In this section we see that we can find the clones on $\Z_p$ with $+$
by another approach using a more abstract result by Hagemann and Herrmann \cite{HH:ALEC}.

Let $\mathbf{A}$ be an algebra.
Then $\alpha \in \Con{\mathbf{A}\times \mathbf{A}}$ is a \emph{product congruence} if there are $\beta, \gamma \in \Con{\mathbf{A}}$ such that 
for all $(v, w), (x,y) \in A \times A$, $(v, w)\alpha(x,y)$ if and only if $v\beta x$ and $w\gamma y$.
Otherwise, we call $\alpha$ a \emph{skew congruence}.
We denote the smallest class of algebras which contains $\mathbf{A}$, and which is closed under
all direct unions, homomorphic images, subalgebras and direct products with finitely many factors
by $\genLocEqCl{\mathbf{A}}$. Such a class is called a \emph{locally equational class}.
An $n$-ary \emph{partial operation} on $A$ is a map $f\colon D \to A$ where $D\subseteq A^n$.
We say that $f$ \emph{preserves} a relation $B \subseteq A^I$ for some index set $I$ if for all
$v_1, \ldots, v_n \in B$, $(f(v_1(i), \ldots, v_n(i)) \mid i \in I) \in B$.
We then write $f \pre B$.

\begin{pro}[{\cite[Proposition 3.2]{HH:ALEC}}] \label{pro:hagher} Let $\mathbf{A}$ be an algebra. Then the following are equivalent:
\begin{enumerate}
 \item $\genLocEqCl{\mathbf{A}}$ is congruence permutable and every subalgebra $\mathbf{B} \leq \mathbf{A}$
 is simple and $\mathbf{B} \times \mathbf{B}$ has no skew congruences.
 \item Every partial operation on $A$ which preserves isomorphisms between subalgebras of $\mathbf{A}$
 can be locally represented by terms.
\end{enumerate} 
\end{pro}

Let $p$ be a prime and let $C$ be a clone on $\Z_p$ which contains $+$ and $C \not \subseteq \affFun{\Z_p}$.
Then we investigate the algebra $\mathbf{A}:= (\Z_p, (f)_{f\in C})$. 
Since $+ \in C$, the algebra $\mathbf{A}$ contains a Mal'cev term ($d(x, y, z) = x-y+z$), 
hence $\mathbb{HSP}(\mathbf{A})$ is congruence permutable (\cite{M:OTGTOAS}) and also congruence modular (\cite{D:CCM}).
We have $\genLocEqCl{\mathbf{A}} \leq \mathbb{HSP}(\mathbf{A})$ (cf.~\cite[Theorem 3.44]{BV:FEL}) and thus $\genLocEqCl{\mathbf{A}}$ is congruence permutable.
If $\mathbf{A}$ contains a constant function different from $0$, then the only subalgebra of $\mathbf{A}$ is $\mathbf{A}$ itself,
otherwise also $(\{0\}, (f)_{f\in C})$ is a subalgebra of $\mathbf{A}$.
Both subalgebras are simple.
Since $\mathbf{A}$ contains a function which is not affine,
the algebra is not affine and thus not abelian (\cite[Corollary 5.9]{FMCK:CT}).
By \cite[Theorem 7.30]{B:UA} $\Con{\mathbf{A}\times \mathbf{A}}$ does not contain a sublattice isomorphic to $M_3$,
hence the algebra $\mathbf{A}$ does not contain any skew congruences.
$(\{0\}, (f)_{f\in C}) \times (\{0\}, (f)_{f\in C})$  obviously does not contain any skew congruences,
since the only congruence is $\{((0,0),(0,0))\}$.
Now Proposition \ref{pro:hagher} leads to a characterization for the clones on $\Z_p$ with $+$ via automorphisms. 
Let $\aut{\Z_p}$ be the automorphism group of $(\Z_p, +)$.   
We know that the set of automorphisms of $(\Z_p, +)$ is given by the functions
  $\pi_a \colon \Z_p \to \Z_p$, $x\mapsto ax$, where $a\in \Z_p\setminus \{0\}$.

We have to distinguish between the cases if $C$ does not preserve $\{(0)\}$ and $C$ preserves $\{(0)\}$.
 
 \begin{cor}\label{cor:hagherCl1}
  Let $p$ be a prime number and
  let $C$ be a clone on $\Z_p$ which contains $+$.
  We assume that $C \not \subseteq \affFun{\Z_p}$ and $C$ does not preserve $\{(0)\}$.
  Then $C$ contains all functions on $\Z_p$.
 \end{cor}

 \begin{proof}
 Let $\mathbf{A}:= (\Z_p, (f)_{f\in C})$.
  The isomorphisms between the subalgebras of $\mathbf{A}$ are the automorphisms of $(\Z_p, +)$.
  Since $C$ does not preserve $\{(0)\}$, the only automorphism which is preserved by $C$ is the identity mapping $\pi_1$.
  Let $g$ be a function on $\Z_p$.
  Then $g$ obviously preserves $\pi_1$.
   Now \ref{pro:hagher} yields that $g$ can be represented by terms, which means that $g \in C$.
 \end{proof}

 \begin{cor}\label{cor:hagherCl2}
  Let $p$ be a prime number and let $C$ be a clone on $\Z_p$ which contains $+$.
  We assume $C \not \subseteq \affFun{\Z_p}$ 
  and $C$ preserves $\{(0)\}$.
  Let $R := \{\pi \in \aut{\Z_p} \mid \forall f \in C \colon f \pre \pi  \}$ and let $g$ be a function on $\Z_p$. 
  If $g$ preserves $\pi$ for all $\pi \in R$ and $g$ preserves $\{(0)\}$,
  then $g \in C$.
 \end{cor}
 
 \begin{proof}
  Let $\mathbf{A}:= (\Z_p, (f)_{f\in C})$.
  The isomorphisms between the subalgebras of $\mathbf{A}$ are the automorphisms of $(\Z_p, +)$ and
  $h \colon \{0\} \to \{0\}$, $h(0) = 0$.
  Since $g$ preserves $\pi$ for all $\pi \in R$ and $g$ preserves $\{(0)\}$, Proposition \ref{pro:hagher} yields that $g$ can be represented by terms, which means that $g \in C$.
 \end{proof}

 Now we get a characterization for the clones on $\Z_p$ which contain $+$, preserve $\{(0)\}$
 and are not a subset of the affine functions. We denote the set of these clones by $\mathbf{C}^+$. 
 Let $R$ be a set of relations on $\Z_p$. 
 We denote the 
 set of all functions on $\Z_p$ that preserve all relations in $R$ by $\Polymorphisms{R}$.
 
  \begin{thm}\label{thm:charClones}
 There is a bijective mapping between $\mathbf{C}^+$ and the set of subgroups of $\aut{\Z_p}$.
 \end{thm}

 \begin{proof}
 We denote the set of subgroups of $\aut{\Z_p}$ by $\mathbb{S}(\aut{\Z_p})$.
 Now we define $\phi \colon \mathbb{S}(\aut{\Z_p}) \to \mathbf{C}^+ $, $R \mapsto \Polymorphisms{R'}$,
 where $R' := R \cup \{\{(0)\}\}$.
 Here $\Polymorphisms{R'}$ is the clone determined by the set
 of binary relations $R$ and the unary relation $\{(0)\}$.
 We show that $\phi$ is an isomorphism.
 Let $R$ be a subgroup of $\aut{\Z_p}$.
 By \cite[Satz 1.2.1]{PK:FUR} $\Polymorphisms{R'}$ is a clone on $\Z_p$.
 
 First, we show surjectivity.
 Let $C \in \mathbf{C}^+$.
 Let $R := \{ \pi \in \aut{\Z_p} \mid \forall f \in C \colon f \pre \pi \} $.
 Now Corollary \ref{cor:hagherCl2} yields that $R \mapsto \Polymorphisms{R'}$.
 It is left to show that $(R, \circ)$ is a subgroup of $\aut{\Z_p}$.
 To this end, it is sufficient to show that $R$ is closed under $\circ$.
 Let $a, b \in \Z_p\setminus \{0\}$ such that $\pi_a, \pi_b \in R$.
 Let $f \in C$.
 We show $f \pre \pi_{a} \circ \pi_b$.
  Let $n \in \N$ be the arity of $f$.
  Since $f \pre \pi_{a}$
  and $f \pre \pi_b$
  we get
 \begin{align*}
 f(\pi_a(\pi_b(x_1)), \ldots, \pi_a(\pi_b(x_n))) &= af(bx_1, \ldots, bx_n) \\
 &= abf(x_1, \ldots, x_n).
 \end{align*}
 Hence $f\pre \pi_a \circ \pi_b$.
 Therefore $\pi_a \circ \pi_b \in R$.
 
 Now we show injectivity of $\phi$.
 To this end, we show for two different subgroups $(R_1, \circ)$ and $(R_2, \circ)$  of $\aut{\Z_p}$, with $R_1 \not \subseteq R_2$
 that $\Polymorphisms{R_1'}\neq \Polymorphisms{R_2'}$.
 Since the automorphism group of $\Z_p$ is cyclic,
 every subgroup of $\aut{\Z_p}$ is cyclic.
  Now let $a, b \in \Z_p\setminus \{0\}$
  such that 
  $\pi_a$ generates $(R_2, \circ)$ and
  $\pi_b$ generates $(R_1, \circ)$.
     Let $\ord{a}$ denote the size of the subgroup of $\aut{\Z_p}$ generated by $\pi_a$.
  Since $R_1 \not \subseteq R_2$ we have 
  $\ord{b} \nmid \ord{a}$.
  Let $g:= x_1\cdots x_{\ord{a}+1}$.
  Now we have $g(ax_1, \ldots, ax_{\ord{a}+1}) = a^{\ord{a}+1}x_1\cdots x_{\ord{a}+1} = a x_1\cdots x_{\ord{a}+1}$.
  Thus $g$ preserves $\pi_a$ and $g$ also preserves $\{(0)\}$.
  Since $g \pre \pi_{a^k}$ for all $k \in \N$,  
  and $\pi_a$ generates $R_2$,
  we have 
  $g \in \Polymorphisms{R_2'}$.
  Now $g(bx_1, \ldots, bx_{\ord{a}+1}) = b^{\ord{a}+1}x_1\cdots x_{\ord{a}+1}$.
  Since $\ord{b} \nmid \ord{a}$, we have  $b^{\ord{a}+1} \neq b$ and hence $g$ does not preserve $\pi_b$ which means that
  $g \not \in \Polymorphisms{R_1'}$.
 \end{proof}
  
  The automorphism group of $\Z_p$ is of size $p-1$.
  All subgroups of $\aut{\Z_p}$ are given by all the divisors of $p-1$.
  Therefore we get by Theorem \ref{thm:charClones} that $\mathbf{C}^+$ is isomorphic to the divisor lattice of $p-1$.
  Since the only clones with $+$ which are subsets of $\affFun{\Z_p}$ are the linear functions and the affine functions,
  we get with the help of Corollary \ref{cor:hagherCl1} again the whole lattice of clones on $\Z_p$ with $+$ which we have already found in Figure \ref{fig:clones}.
  
  \section{A result for clones on $\Z_p \times \Z_p$}\label{sec:ZpZp}
  
  Let $p > 2$ be a prime number.
  The goal of this section is to show that we can embed the linearly closed clonoids on $\Z_p$ into clones on $\Z_p \times \Z_p$
  and obtain infinitely many non finitely generated clones on $(\Z_p \times \Z_p, \oplus)$, where $\oplus$ is the componentwise addition.
  Let $n \in \N$.
  Then we write $\bar{x}^n$ for $(x_1, \ldots, x_n) \in \Z_p^n$. 
  
  \begin{de}
   Let $p$ be a prime number and let $C$ be a linearly closed clonoid on $\Z_p$.
   Then we define
    \begin{align*}
     \phi(C) := \{(\Z_p^2)^n\to\Z_p^2, &\big((x_1, y_1), \ldots ,(x_n, y_n)\big) \mapsto \big(l_1(\bar{x}^n) + f(\bar{y}^n), l_2(\bar{y}^n) \big)\mid \\ 
     &n \in \N, l_1, l_2 \in \linFun{\Z_p}^{\ar{n}},  f\in C^{\ar{n}} \}.
    \end{align*} 
  \end{de}
  \noindent
  \emph{Remark}: A more general embedding was introduced in a manuscript of Stefano Fioravanti \cite{F:TE}.
  
  \begin{pro} \label{pro:resultZpZp}
  Let $p$ be a prime number and let $C, D$ be linearly closed clonoids on $\Z_p$.
   Then $\phi(C)$ is a clone on $\Z_p \times \Z_p$ which contains $\oplus$.
   Furthermore, $C \subseteq D$ if and only if $\phi(C) \subseteq \phi(D)$.
  \end{pro}
  
  \begin{proof}
   First we show that $\phi(C)\subseteq \bigcup \{(\Z_p \times \Z_p)^{(\Z_p\times \Z_p)^i}\mid i \in \N \}$ is a clone on $\Z_p\times \Z_p$.
   To this end we show that the projections lie in $\phi(C)$
   and $\phi(C) \phi(C) \subseteq \phi(C)$.

   Let $n\in \N$. 
   We have $0^{\Z_p^n} \in C$,
   and the projection  $\pi_j^n \in \linFun{\Z_p}$ for all $j \leq n$.
   Hence
   $\Pi_j^n\colon(\Z_p^2)^n \to \Z_p^2$,
   $\big((x_1, y_1), \ldots (x_n, y_n)\big) \mapsto (x_j, y_j)$ lies in $\phi(C)$
   for all $j \leq n$.

   Now we show $\phi(C)\phi(C)\subseteq \phi(C)$.
   To this end, let $n, m \in \N$, let $g \in \phi(C)^{\ar{n}}$ and let $f_1, \ldots, f_n \in \phi(C)^{\ar{m}}$.
   Now we show $g(f_1, \ldots, f_n) \in \phi(C)$.
   For $i \leq n$, let $l_{i}, l_{i}' \in \linFun{\Z_p}$ and $h_i \in C$
   such that $$f_i\big((x_1, y_1), \ldots, (x_m, y_m)\big) = \big(l_{i}(\bar{x}^m) + h_i(\bar{y}^m), l_{i}'(\bar{y}^m)\big).$$
   Now let $k_1, k_2 \in \linFun{\Z_p}$ and $g' \in C$ such that 
   $$g\big((x_1, y_1), \ldots, (x_n, y_n)\big) = \big(k_1(\bar{x}^n) + g'(\bar{y}^n), k_2(\bar{y}^n)\big).$$
   Now we have 
    \begin{align*}
     &\big(g(f_1, \ldots, f_n)\big)\big((x_1, y_1), \ldots,(x_m, y_m)\big)   \\ 
     & = g\Big(f_1\big((x_1, y_1), \ldots,(x_m, y_m)\big), \ldots, f_n\big((x_1, y_1), \ldots,(x_m, y_m)\big)\Big) \\
     & = g\Big(\big(l_{1}(\bar{x}^m) + h_1(\bar{y}^m), l_{1}'(\bar{y}^m)\big), \ldots, \big(l_{n}(\bar{x}^m) + h_n(\bar{y}^m), l_{n}'(\bar{y}^m)\big) \Big) \\
     & = \bigg(k_1\Big(\big(l_{1}(\bar{x}^m) + h_1(\bar{y}^m)\big),\ldots, \big(l_{n}(\bar{x}^m) + h_n(\bar{y}^m) \big)\Big) + g'\big(l_{1}'(\bar{y}^m), \ldots, l_{n}'(\bar{y}^m)\big), \\ 
     & \hspace{1cm}k_2 \big( l_{1}'(\bar{y}^m), \ldots,  l_{n}'(\bar{y}^m)\big) \bigg) \\    
     & = \Big(k_1\big(l_{1}(\bar{x}^m), \ldots,l_{n}(\bar{x}^m)\big) + k_1\big(h_1(\bar{y}^m),\ldots, h_n(\bar{y}^m)\big) \\
     &  \hspace{4,8cm} + g'\big(l_{1}'(\bar{y}^m), \ldots, l_{n}'(\bar{y}^m) \big) ,
      k_2 \big( l_{1}'(\bar{y}^m), \ldots, l_{n}'(\bar{y}^m) \big) \Big).
    \end{align*}
   
   Since $h_1, \ldots, h_n \in C$ and $g' \in C$
   we have by the axioms of a linearly closed clonoid that $k_1\big(h_1(\bar{y}^m),\ldots, h_n(\bar{y}^m)\big) + g'\big(l_{1}'(\bar{y}^m), \ldots, l_{n}'(\bar{y}^m) \big)\in C$.
   We also have $\linFun{\Z_p} \linFun{\Z_p} \subseteq \linFun{\Z_p}$ and thus 
  we see that $g(f_1, \ldots, f_n)$ lies in $\phi(C)$.
  
  Obviously, $\oplus$ lies in $\phi(C)$.
  This concludes the proof that $\phi(C)$ is a clone on $\Z_p \times \Z_p$ which contains $\oplus$.
  
  Now we show $C \subseteq D$ if and only if $\phi(C) \subseteq \phi(D)$.
  
  ``$\Rightarrow$'': Holds by the definition of $\phi$.
 
  ``$\Leftarrow$'':
  Let $n \in \N$ and let $f \in C^{\ar{n}}$. 
  Then $\bar{f}\colon((x_1, y_1), \ldots, (x_n, y_n)) \mapsto (0^{\Z_p^n}(\bar{x}^n)+f(\bar{y}^n), 0^{\Z_p^n}(\bar{y}^n)) \in \phi(C)$ since $0^{\Z_p^n} \in \linFun{\Z_p}^{\ar{n}}$.
  Then $\bar{f} \in \phi(D)$ and thus $f \in D$. 
  \end{proof}

  \begin{thm} 
     Let $p > 2$ be a prime number.
     Then there are infinitely many non finitely generated clones on $\Z_p \times \Z_p$ which contain $\oplus$. 
  \end{thm}
  
  \begin{proof}
  Let $C$ be a linearly closed clonoid which is not finitely generated.
  Then let $(C_i)_{i\in\N}$ be an infinite ascending chain of linearly closed clonoids on $\Z_p$ such that $C = \bigcup \{C_i \mid i \in \N\}$.
  By Proposition \ref{pro:resultZpZp} $\phi$ is an embedding from the lattice of linearly closed clonoids on $\Z_p$ to
  the lattice of clones on $\Z_p \times \Z_p$,
  and thus
  $(\phi(C_i))_{i\in\N}$ is an infinite ascending chain of clones on $\Z_p \times \Z_p$ which contain $\oplus$.
  Then 
  $\phi(C) = \phi(\bigcup \{C_i \mid i \in \N\})  = \bigcup \{\phi(C_i) \mid i \in \N\}$
  is not finitely generated.
  By Proposition \ref{pro:infNonFinGen} we know that there are infinitely many non finitely generated linearly closed clonoids on $\Z_p$
  and thus there are infinitely many non finitely generated clones on $\Z_p \times \Z_p$ which contain $\oplus$.
  \end{proof}

  \noindent
  \emph{Remark:}
  The generalization from $\Z_p$ to any abelian group of prime power order is not straightforward,
  since in the case of an abelian group of prime power order the clone generated by $+$ does not contain
  all possible endomorphisms.
  So the proof of Lemma \ref{lem:redLemPol} does not work for an abelian group of prime power order in general. 

  \section*{Acknowledgements}
The author thanks Erhard Aichinger, who has inspired this work, and 
Stefano Fioravanti for many hours of helpful discussions. 
The author is also grateful to the anonymous referee
for a very careful report.

\providecommand{\bysame}{\leavevmode\hbox to3em{\hrulefill}\thinspace}
\providecommand{\MR}{\relax\ifhmode\unskip\space\fi MR }

\providecommand{\MRhref}[2]{%
  \href{http://www.ams.org/mathscinet-getitem?mr=#1}{#2}
}
\providecommand{\href}[2]{#2}

\end{document}